\documentclass[12pt, reqno]{amsart}
\usepackage{amssymb, amsmath, amsthm}
\usepackage[backref, hidelinks]{hyperref}
\usepackage[alphabetic,backrefs,lite]{amsrefs}
\usepackage{amscd}   
\usepackage{fullpage}
\usepackage[all]{xy} 
\usepackage{setspace}
\usepackage{mathtools}
\usepackage{enumerate}

\DeclareFontEncoding{OT2}{}{} 

\usepackage[usenames,dvipsnames]{color}

\newtheorem{lemma}{Lemma}[section]
\newtheorem{theorem}[lemma]{Theorem}
\newtheorem{proposition}[lemma]{Proposition}

\newtheorem{cor}[lemma]{Corollary}

\newtheorem{claim*}{Claim}

\theoremstyle{definition}
\newtheorem{remark}[lemma]{Remark}
\newtheorem{notational}[lemma]{Notational Convention}

\newcommand{\C}{{\mathbb C}}

\newcommand{\calC}{{\mathcal C}}
\newcommand{\calD}{{\mathcal D}}

\newcommand{\calF}{{\mathcal F}}
\newcommand{\calG}{{\mathcal G}}

\newcommand{\calK}{{\mathcal K}}

\newcommand{\calM}{{\mathcal M}}

\newcommand{\calO}{{\mathcal O}}

\newcommand{\calU}{{\mathcal U}}

\newcommand{\calX}{{\mathcal X}}

\newcommand{\defeq}{\coloneqq}

\DeclareMathOperator{\rk}{rk}

\DeclareMathOperator{\Hom}{Hom}

\DeclareMathOperator{\Pic}{Pic}

\DeclareMathOperator{\id}{id}

\DeclareMathOperator{\GCD}{GCD}

\numberwithin{equation}{section}
\numberwithin{table}{section}

\def\calC{\mathcal{C}}
\def\calD{\mathcal{D}}

\def\calF{\mathcal{F}}
\def\calG{\mathcal{G}}

\def\calK{\mathcal{K}}

\def\calM{\mathcal{M}}

\def\calO{\mathcal{O}}

\def\calU{\mathcal{U}}

\def\calX{\mathcal{X}}

\def\bC{\mathbf{C}}

\def\bP{\mathbf{P}}
\def\bQ{\mathbf{Q}}
\def\bR{\mathbf{R}}

\def\bZ{\mathbf{Z}}

\begin{document}
\title{Kodaira dimension of moduli of special $K3^{[2]}$-fourfolds of degree 2}
\author{Jack Petok}
\date{January 1, 2023}

\address{Department of Mathematics, Dartmouth College, Kemeny Hall, Hanover, NH 03755}
\email{jack.petok@dartmouth.edu}
\urladdr{http://math.dartmouth.edu/\~{}jpetok}

\begin{abstract}
We study the Noether-Lefschetz locus of the moduli space $\mathcal{M}$ of $K3^{[2]}$-fourfolds with a polarization of degree $2$. Following Hassett's work on cubic fourfolds, Debarre, Iliev, and Manivel have shown that the Noether-Lefschetz locus in $\mathcal{M}$ is a countable union of special divisors $\mathcal{M}_d$, where the discriminant $d$ is a positive integer congruent to $0,2,$ or $4$ modulo 8. We compute the Kodaira dimensions of these special divisors for all but finitely many discriminants; in particular, we show that for $d>224$ and for many other small values of $d$, the space $\calM_d$ is a variety of general type.

{\bf R\'{e}sum\'{e} } On \'{e}tudie le lieu de Noether-Lefschetz dans l'espace de modules $\mathcal{M}$ des vari\'{e}t\'{e}s de type $K3^{[2]}$ munies des polarisations de degr\'{e} $2$. Selon l'approche de Hassett pour les cubiques de dimension quatre, Debarre, Iliev, et Manivel ont \'{e}tablit que ce lieu dans $\mathcal{M}$ est une r\'{e}union des diviseurs sp\'{e}ciaux $\mathcal{M}_d$, o\`{u} le discriminant $d$ est un entier positif congru \`{a} $0,2$, ou $4$ modulo $8$. On calcule les dimensions de Kodaira des diviseurs sp\'{e}ciaux pour presque tous les discriminants; en particulier, on d\'{e}monstre que, pour $d >224$ et des autres petits entiers $d$, l'espace $\calM_d$ est une vari\'{e}t\'{e} de type g\'{e}n\'{e}ral.
\end{abstract}

\maketitle

\section{Introduction}
\label{S:intro}

The aim of this paper is to study the internal geometry of some moduli spaces of hyperk\"{a}hler fourfolds. Let $\calM$ denote the moduli space of complex four-dimensional polarized hyperk\"{a}hler (HK) manifolds of $K3^{[2]}$ type with polarization of degree $2$, the simplest possible polarization degree. The variety $\calM$, quasi-projective and of dimension 20, is also the period space for Gushel--Mukai fourfolds, as well as the period space for EPW double sextics. A very general $X \in \calM$ has the property that $X$ has Picard rank 1. The locus where this property fails is the \textit{Noether-Lefschetz locus} $NL(\calM)$ of $\calM$:
$$NL(\calM) = \{(X,H) \in \calM(\bC) \ \colon \ \rk \Pic X \ge 2\},$$
which is a union of  countably many irreducible divisors known as the  (Noether-Lefschetz)-{\it special divisors} in $\calM$. Our specific goal in this paper is the computation of the Kodaira dimensions of these special divisors.

\subsection{Statement of main theorem} 
\label{ss:mainthms}
Recall that for any HK manifold $X$, the Picard group $\text{Pic X}$ injects (via the exponential exact sequence) into the singular cohomology group $H^2(X, \bZ)$. The \textit{Beauville-Bogomolov form} $q_X \colon H^2(X, \bZ) \to \bZ$ equips $H^2(X, \bZ)$ with the structure of an even integral lattice. A point $p \in \calM$ is represented by a pair $(X,H)$ where $X$ is an HK fourfold of deformation type $K3^{[2]}$ and $H \in \text{Pic}(X)\hookrightarrow H^2(X, \bZ)$ is an ample divisor with $q_X(H) =H^2=2.$
 A polarized HK fourfold $(X,H)$ is said to be \textit{special} if $(X,H)  \in NL(\calM)$. A primitive sublattice $K\subseteq \Pic X$ of rank 2 containing $H$ forms the data of a \textit{special labelling of discriminant $d$} for $X$ (or more precisely, for $(X,H)$), where $d=|D(K_{H^2(X,\bZ)}^\perp)|$ (cf.~\cite[\S 4]{DM}).

For each $d$, there is a moduli space $\calM_d \subset \calM$ of polarized special $K3^{[2]}$-fourfolds of discriminant $d$. The nonempty $\calM_d$ are hypersurfaces in $\calM$, first studied by Debarre, Iliev, and Manivel in~\cite{DIM} as the locus of Hodge structures possessing a special discriminant $d$ labelling in the period domain for prime Fano fourfolds of index $10$ and degree $2$ (such Fano fourfolds are also known as {\it Gushel--Mukai fourfolds}). They prove that the moduli space $\calM_d$ is nonempty if and only if $d \notin \{2,8\}$ and $d \equiv 0, 2, 4 \bmod 8$. Furthermore, the divisor $\calM_d$ is irreducible if $d \equiv 0,4 \bmod 8$ or $d=10$; otherwise, when $d\equiv 2 \bmod 8$, the hypersurface $\calM_d$ of special fourfolds of discriminant $d$ is the union of two irreducible divisors, denoted $\calM_d'$ and $\calM_d''$, which are birationally isomorphic (see~\cite[Theorem 6.1]{DM}).

In this paper, we determine the Kodaira dimension of $\calM_d$ for nearly every value of $d$. We show $\calM_d$ is of general type for almost all $d$:
$$d > 224 \implies \kappa(\calM_d) = 19.$$
 Moreover, we push our methods to determine the Kodaira dimension for many other small values of $d$. Our results, together with the additional inputs to be discussed in~\S\ref{relations}, determine information about the birational type of $\calM_d$ for all but 34 discriminants.\footnote{The 34 discriminants for which we have no information on the Kodaira dimension of $\calM_d$ at the present time are: $12, 16, 18, 24, 28, 32, 36, 40, 42, 48, 50, 52, 56, 58, 60, 64, 66, 68, 72,
74, 76, 80, 82, 84, 90, 92, \newline100, 108, 112, 114, 124, 128, 130, 
176.$}

Our goal is to prove the following theorem:
\begin{theorem}
\label{thm:MainThm}
Let $\calM$ denote the moduli space of hyperk\"{a}hler fourfolds of degree 2 of $K3^{[2]}$-type, and let $\calM_d \subset \calM$ denote the moduli space of special $K3^{[2]}$-fourfolds with a special labelling of discriminant $d$.  
\medskip
\begin{enumerate} 
\item Suppose that $d = 8m$ with $m \ge 11$. Then $\calM_d$ is of  general type for  $m\notin  \{11,12,13,14, \newline 16, 17, 22,25,28\}$.  Furthermore, for $m \notin \{14,16, 22 \}$, the variety $\calM_d$ has nonnegative Kodaira dimension.\medskip

\item Suppose that $d = 8m+2$ with $m \ge 12$. Then $\calM_{d}$ has two birationally isomorphic irreducible components, $\calM_d'$ and $\calM_d''$, both of which are of general type when \newline $m \notin \{12,13,14,15,16,17,21,23\}$. Furthermore, for $m \notin \{14,16\}$, the varieties $\calM_d'$ and $\calM_d''$ have nonnegative Kodaira dimension.\medskip

\item Suppose that $d = 8m +4$ with $m \ge 14$. Then $\calM_d$ is of general type if \newline $m \notin \{15,17,21,25,27\}$. Furthermore, for $m \ne 15$, the variety $\calM_d$ has nonnegative Kodaira dimension.\medskip

\end{enumerate}
\end{theorem}

The idea of the proof is to work with the global period domain $\calD_d$, an irreducible quasi-projective variety. The Torelli theorem for $\calM$ shows that $\calM_d$ is a Zariski open subset of $\calD_d$. Then we use automorphic techniques developed by Gritsenko-Hulek-Sankaran in~\cite{GHSInventiones} and~\cite{GHSHandbook} to study the Kodaira dimension of $\calD_d$. This requires the construction of special odd weight modular forms on certain quotients of type IV Hermitian symmetric domains of the form $\widetilde{O}^+(L) \backslash \Omega_L^+$ (see~\S\ref{s:basics} for the relevant definitions). 

We note that by a result of Ma, there are only finitely many even integral lattices $L$ of signature $(2,n)$ such that $\widetilde{O}^+(L) \backslash \Omega_L^+$ is {\it not} of general type (\cite[Theorem 1.3]{Ma}).  Ma's result implies that each nonempty $\calD_d$ is of general type for $d\ge D_0$, where $D_0$ is some constant $D_0 \ge 5.5 \cdot 10^{16}$. In the present work, we find a smaller upper bound, $d_0 = 224$, such that each nonempty $\calD_d$ is of general type for $d > d_0$.

\subsection{Relationship to $\calK_d$ and $\calC_d$}\label{relations} There are 40 values of $d$ for which the techniques used to prove Theorem~\ref{thm:MainThm} do not yield any information about $\calM_d$. However, it is possible to use results on the Kodaira dimension of the moduli space of degree $d$ polarized $K3$ surfaces $\calK_d$ to conclude something about $\calM_d$ for some of these discriminants. For $d=2k$ with $1 \le k \le 13$ or $ k \in \{15, 16, 17, 19\}$, it is known that $\calK_d$ has negative Kodaira dimension, and in fact $\calK_d$ is unirational~(\cite[Theorem 4.1]{GHSHandbook} and~\cite{Nuer}). Since $\calK_d$ dominates $\calM_d$ whenever $d$ is not divisible by a prime $3 \bmod 4$ and $\calM_d \ne \emptyset$~(\cite[Proposition 6.5]{DIM}), we conclude that $\calM_d$ has negative Kodaira dimension and is in fact unirational when $d \in \{4, 10, 20, 26, 34\}.$ 

Similarly, the moduli space $\calC_d$ of special cubic fourfolds of discriminant $d$ dominates $\calM_d$ whenever $d \equiv 2$ or $20 \bmod 24$ and the only odd primes  dividing $d$ are congruent to $\pm 1\bmod 12$~(\cite[Proposition 6.5]{DIM}). The only new information this yields about the Kodaira dimension of $\calM_d$ is that $\calM_{44}$ has negative Kodaira dimension, since $\calC_{44}$ is uniruled by work of Nuer (see~\cite{Nuer}).  
\subsection{EPW double sextics and $\calM_d$}
O'Grady has shown that a general $(X,H)\in \calM$ is a smooth EPW double sextic (see~\cite{O'Grady}). Precisely, there is a Zariski open subset $\calU$ of $\calM$ parametrizing pairs $(X,H)$ with ample and base-point free $H$ such that $|H| \colon X \to \bP^5$ realizes $X$ as a ramified double cover of an EPW sextic in $\bP^5$. We can consider the subvariety $\calU_d=\calM_d \cap \calU \subset \calM$ in $\calU$ parametrizing EPW double sextics which have a special labelling of discriminant $d$. It is possible (see~\cite[Example 6.3]{DM}) that $\dim \calU_d < \dim \calM_d$: if $d=4$ then $\calU_d =\emptyset$, and while $\calU_d$ is known to be nonempty for $d\ge 10$ and $d \equiv 0,2,4 \mod 8$, it is unknown whether $\dim \calU_d = \dim \calM_d$ for such $d$. Still, for $d$ sufficiently large, the variety $\calU_d$ is birational to $\calM_d$ (because $\calU$ is an open subset of $\calM$), and thus we can conclude that $\calU_d$ is of general type for such $d$. It would follow from a conjecture of O'Grady~\cite[Example 6.3]{DM} that $\calU_d$ is birational to $\calM_d$ for all $d\ne 4$.

\begin{cor}
\label{cor:EPWCor}
Let $\calU_d$ denote the moduli space of smooth EPW double sextics that possess a special labelling of discriminant $d$. Then for all sufficiently large $d$ the following hold:
\begin{itemize}
\item If $d\equiv 0,4 \bmod 8$, then the space $\calU_d$ is of general type.
\item If $d \equiv 2 \bmod 8$, then both irreducible components of $\calU_d$ are of general type.
\end{itemize}
\end{cor}
If O'Grady's conjecture is true, then one can take $d >224$ in Corollary~\ref{cor:EPWCor}, but as of this writing the result remains ineffective.
 \begin{remark}
There is an remarkable geometric association, first appearing in~\cite{IM}, between Gushel--Mukai fourfolds and EPW double sextics, which gives a morphism from the 24-dimensional moduli stack of GM fourfolds to the 20-dimensional moduli stack of EPW double sextics; in particular, the image of a special Gushel--Mukai fourfold of discriminant $d$ is a special EPW double sextic of discriminant $d$~(cf.~\cite{DIM},~\cite{DK}), and hence the image of the locus of special Gushel--Mukai fourfolds lies in $\calU_d$.
 \end{remark}

\subsection{Overview}
In~\S\ref{s:basics} we review some relevant notions about lattices and hyperk\"{a}hler varieties. Then we give the definitions of the moduli spaces $\calM$ and $\calM_d$, and explain how the work of Gritsenko--Hulek--Sankaran determines Kodaira dimension of these varieties provided a modular form can be constructed with special properties. The strategy is to build modular forms using a kind of ``pulling back'' of the Borcherds modular form $\Phi_{12}$. For this, we need to construct special lattice embeddings.

The systematic study of these lattice embeddings is taken up in~\S\ref{S:constructingembeddingsgeneralities}. Here, we use a slightly modified version of the ``lattice engineering'' trick from~\cite[Section 4]{TVACrelle}. We formulate elementary conditions on certain lattice embeddings  from which Theorem~\ref{thm:MainThm} will follow. 

In~\S\ref{S:constructingembeddingsspecifics}, we take up actual construction of these embeddings with the desired properties, breaking our analysis into the cases $d=8m$, $8m+2$, and $8m+4$ (see ~\S\ref{ss:d8m}, ~\S\ref{ss:d8m+2}, ~\S\ref{ss:d8m+4}). We
then reduce the problem of constructing special embeddings to a number theoretic problem concerning the integer valued points on a diagonal quadric. To guarantee the existence of such points for all sufficiently large $d$, we invoke a
classical result of Halter-Koch on the sums of three squares. The final part of the argument deals with the low values of the discriminant $d$ using computer code code written in the {\tt Magma} language~\cite{Magma}, which is provided on the author's webpage.

\section{Basic notions and definitions}
\label{s:basics}
 In this section we define the main objects of the paper, starting with a review of lattice theory in~\S\ref{ss:lattices} and the moduli and periods of our hyperk\"{a}hler fourfolds in~\S\ref{ss:moduliandperiods}. The special divisors $\calM_d$ and $\calD_d$ are discussed in~\S\ref{ss:NL}, and the orthogonal modular varieties $\calF_d$ are discussed in~\S\ref{ss:orthogmodvars}.

 \subsection{Lattices} \label{ss:lattices} (References:\cite{CS},~\cite{SerreCourse}.) An (integral) {\it lattice} is a free $\bZ$-module $L$ of finite rank together with a nondegenerate symmetric $\bZ$-bilinear form 
$$(\cdot, \cdot)\colon L \times L  \to \bZ.$$
The {\it signature} $(r,s)$ of $L$ is the signature of a Gram matrix for $L$. A lattice $L$ is {\it even} if $(x,x) \defeq x^2  \in 2\bZ$ for all $x \in L$. An element $x \in L$ is {\it primitive} if it is not an integral multiple of any other vector in $L$. An $(n)$-root of $L$ is any primitive vector $r$ of square-length $r^2=n$. 

An embedding $L \hookrightarrow M$ of integral lattices is {\it primitive} if the quotient group $M/L$ is torsion-free. The orthogonal complement of $L$ in $M$ will be denoted $L^\perp_M$, or simply $L^\perp$ with the ambient lattice understood from context. To every even integral lattice $L$, there is the associated dual lattice $L^\vee = \Hom(L, \bZ)$ with an embedding $L \hookrightarrow L^\vee$ given by $x \mapsto (x, \cdot)$. The group $D(L) \defeq L^\vee/L$ is a finite abelian group, called the {\it discriminant group}. The natural extension of $(\cdot, \cdot)$ to $L^\vee$ endows $L^\vee$ with a $\bQ$-valued bilinear form,. which in turn gives rise to a $\bQ/2\bZ$-valued bilinear form $b_L$ on $L^\vee/L$, called the {\it discriminant form}. An integral lattice is {\it unimodular} if it has trivial discriminant group. Let $O(L)$ denote the group of automorphisms of $L$ preserving $(\cdot, \cdot)$, and let $\widetilde{O}(L)$ denote the subgroup of automorphisms which preserve the discriminant form; that is,
$$\widetilde O(L) \defeq \ker(O(L) \to O(L^\vee/L)).$$
The group $\widetilde{O}(L)$ is a finite index subgroup of $O(L)$ and is known as the {\it stable orthogonal group}. 
In this work, the notation $(n)$ for a nonzero integer $n$ will denote a rank 1 integral lattice with a generator $x$ of length $n$. Following standard practice, the lattice $A_1$ denotes the lattice $(2)$.  If $L$ is a lattice, then $L(n)$ denotes the lattice with the same underlying abelian group as $L$ with pairing given by 
$$(x,y)_{L(n)} = n\cdot (x,y)_L.$$

 Often, we will write down a lattice by writing down a Gram matrix for a basis of the lattice. The lattices $U$ and $E_8$ denote, respectively, the hyperbolic plane given by the Gram matrix $\begin{pmatrix} 0 & 1 \\ 1 & 0 \end{pmatrix}$, and the unique unimodular positive-definite even lattice of rank 8. Later, when perform explicit computation involving $E_8$, we make use of the Gram matrix for $E_8$~(\cite[Ch 4, \S 8]{CS}):
  $$E_8 \cong \begin{pmatrix} 2 &  0 & -2 & -1 &  0 &  0 &  0 & 0 \\
 0&2&0 &-1& -1&0&0&0 \\
-2&0&4&0&0&0&0&1 \\
-1 &-1&0&2&0&0&0&0 \\
 0& -1&0&0&2 &-1&0&0 \\
 0&0&0&0 &-1&2& -1&0 \\
 0&0&0&0&0& -1&2&0 \\
0&0&1&0&0&0&0&2 
\end{pmatrix}.$$
We also need the ``checkerboard'' lattice $D_6$~(\cite[\S 7]{CS}): let $e_1, \ldots, e_6$ denote the standard basis of $\bZ^6 \subset \bR^6$ with the usual dot product. Then we define an even integral lattice $D_6$ by $D_6=\{\sum c_i e_i \in \bZ^6 \ \colon \ \sum c_i \equiv 0 \bmod 2\} \subset \bZ^6$. The $2$-roots of $D_6$ (i.e. the square-length $2$ vectors) are given by $S \cup -S$, where $S=\{e_i \pm e_j \ \colon \ i \ne j\}$. The dual lattice $D_6^\vee$ is the $\bZ$-span of $\bZ^6$ and the vector $(\frac{1}{2}, \frac{1}{2}, \frac{1}{2}, \frac{1}{2}, \frac{1}{2}, \frac{1}{2})$.

\begin{remark} \label{rem:e8sublatices} If $A_1^{\oplus 2} \hookrightarrow E_8$ a is primitive embedding of lattices, then $(A_1^{\oplus 2})^\perp \cong D_6$. This can be verified by direct computation, first on a single embedding, and then by using that embeddings $A_1^{\oplus 2} \hookrightarrow E_8$ are unique up to isometry (see~\cite[Theorem 1.14.4]{Nikulin}). 
\end{remark}

When $L$ has signature $(2,m)$, we also define the subgroup $O^+(L)$ of automorphisms which preserve the orientation on the positive-definite part of $L$. Note that $O^+(L)$ is a finite index subgroup of $O(L)$ and that $O^+(L)$ acts on the period space for $L$: 
$$\Omega_L^+ \defeq \{x \in \bP(L\otimes \bC) \ \colon \ (x,x)=0, (x,\overline{x})>0 \}^+$$
where the $+$ notation indicates that we are taking one component of the two-component set $\{x \in \bP(L\otimes \bC) \ \colon \ (x,x)=0, (x,\overline{x})>0 \}$ (the two components are exchanged by complex conjugation). For any primitive vector $r\in L$ of square length $r^2<0$,  there is a \textit{rational quadratic divisor} in $\Omega_L^+$ defined by 
$$\Omega^+_{L}(r) \defeq \{Z \in \Omega_L^+ \ \colon \ (Z,r) = 0 \}.$$
We will also need the group
$$\widetilde{O}^+(L) \defeq O^+(L) \cap \widetilde{O}(L)$$
which is a finite index subgroup of the groups $O(L), O^+(L)$, and $\widetilde{O}(L)$, and acts properly and discontinuously on $\Omega_L^+$ (as does any finite index subgroup $\Gamma \subseteq O^+(L)$). 
For a sublattice $K \subset L$, define 
$$O(L,  (K)) =  \{g \in O(L) \ \colon \ g(K) = K\}$$
and define
$$O(L, K) = \{g \in O(L, (K) ) \  \colon \ g|_K = \id_K\}.$$
We will write $O(L,v) \defeq O(L, \bZ v)$ for $v \in L$. One can also define $O^+(L,  (K))$, $\widetilde{O}^+(L, K)$, and so on.

\subsection{Moduli and periods of hyperk\"{a}hler fourfolds of $K3^{[2]}$-type}\label{ss:moduliandperiods} (Reference: \cite{Debarre}). Let $X$ be a complex algebraic variety which is deformation equivalent to the Hilbert scheme $S^{[2]}$ of length-two zero-dimensional subschemes of a $K3$ surface $S$ (a variety {\it of $K3^{[2]}$-type}). Then $X$ is a four-dimensional hyperk\"{a}hler (HK) manifold --- meaning $X$ is simply connected with a nowhere degenerate $2$-form $\omega$ such that $H^0(X, \Omega^2_X) = \bC \omega$. Any HK manifold has $H^r(X, \calO_X) = 0$ for any $r$ odd, so the exponential exact sequence shows that $\Pic X$ injects into $H^2(X, \bZ)$. The second integral singular cohomology also underlies a Hodge structure of weight 2 of $K3$-type. The gives another realization of the Picard group as $\Pic X = H^{1,1}(X) \cap H^2(X, \bZ)$.  

The group $H^2(X, \bZ)$ (and its subgroup $\Pic X$) inherits the structure of a quadratic space  from the \textit{Beauville-Bogomolov-Fujiki (BBF) form} $q_X$, a certain canonically defined nondegenerate integral quadratic form of signature $(3, b_2(X) -3)$. For more on $q_X$ we refer the reader to~\cite{Beauville}. For $S$ a $K3$ surface, the second cohomology with the BBF form $(H^2(S^{[2]}, \bZ), q_S)$ is isomorphic to $H^2(S, \bZ) \oplus \bZ \delta$ with $\delta^2 = -2$. The summand $H^2(S, \bZ)$ is the K3 lattice and carries an intersection form given by the cup product, with $s \cdot s = q(s)$. The class $2\delta$ is corresponds to the divisor in $S^{[2]}$ parametrizing nonreduced subschemes of $S$ of length two.  Since $q(H^2(S^{[2]}, \bZ)) = 2\bZ$, the cohomology group $H^2(S^{[2]}, \bZ)$ has the structure of an even, integral lattice.

 The second integral cohomology with the BBF form is deformation invariant. As $H^2(S, \bZ) \cong   U^{\oplus 3} \oplus E_8(-1)^{\oplus 2} $ for any $K3$ surface $S$, it follows for $X$ a fourfold of $K3^{[2]}$-type that $H^2(X, \bZ)$ is isomorphic to the lattice
$$M = U^{\oplus 3} \oplus E_8(-1)^{\oplus 2}  \oplus (-2).$$
Let $u,v$ denote a null basis for the first copy of $U$  in the decomposition of $M$:
$$u^2=v^2=0, (u,v)=1.$$
Let $u',v'$ denote a null-basis for the second copy of $U$, and let $w$ denote the $(-2)$ factor in the decomposition above.

A {\it polarized} HK fourfold is a pair $(X, H)$ where $H \in \Pic X$ is a primitive, ample divisor with $q(H) =e>0$. The integer $e$ is called the {\it degree} of the polarized fourfold. In this work we consider the lowest possible polarization degree $K3^{[2]}$-type fourfolds, those with degree $e=2$. There is a coarse quasi-projective moduli space $\calM$, which is irreducible and has dimension 20, parametrizing polarized $K3^{[2]}$-type fourfolds of degree 2 up to isomorphism; O'Grady showed that this moduli space is unirational (see~\cite[Theorem 1.1]{O'Grady}). A {\it marking} of an HK fourfold of $K3^{[2]}$-type is an isomorphism 
$$\varphi \colon H^2(X, \bZ) \cong M.$$
Every marking on some $(X,H) \in \calM$ is equivalent, under $O(M)$, to one sending $H$ to $h \defeq u+v$. One computes that 
$$h^\perp = \Lambda \defeq  U^{\oplus 2} \oplus E_8(-1)^{\oplus 2}  \oplus (-2)^{\oplus 2}.$$

We briefly recall some relevant Hodge theory for our degree 2 $K3^{[2]}$-fourfolds. The {\it period} of a point $(X, H) \in \calM$ together with marking $\varphi$ is the line
$$\varphi_\bC(H^{2,0}(X)) \in \Lambda \otimes \bC.$$ 
A period determines, via the Hodge-Riemann relations, a weight 2 Hodge structure on $\Lambda$ of $K3$-type.  The global and local period domains for $\Lambda$ are spaces that parametrize these Hodge structures. There exists a map to the {\it local period domain} $\Omega_\Lambda^+$,
$$\{(X, H, \varphi) \ \colon \ (X,H) \in \calM,  \ \varphi \colon H^2(X, \bZ) \to M, \ \varphi(H) = h\}   \longrightarrow \Omega_\Lambda^+,$$
which sends a triple $(X, H, \varphi)$ to its period; after quotienting out by isomorphism of these triples, one gets a map into the {\it global period domain}
$$\tau \colon \calM \to \calD \defeq \widetilde{O}^+(\Lambda) \backslash  \Omega_\Lambda^+.$$
Applying well-known results of Baily-Borel~\cite{BB}, the arithmetic quotient $\calD$ is a quasi-projective, irreducible, normal variety. Using Markman's computation on the monodromy of $K3^{[n]}$-type manifolds~(\cite[Theorem 1.2]{MarkmanK3n}), we see that the the group $\widetilde{O}^+(\Lambda)$ is the monodromy group generated by parallel-transport operators respecting the polarization.
Hence, by the global Torelli theorem for polarized HK fourfolds, due to Verbitsky and Markman (see~\cite[Theorem 8.4]{Markman}), the morphism $\tau$ is algebraic and is an open immersion. 
We note for later use that 
$$\widetilde{O}^+(\Lambda) = \{\gamma \in O^+(\Lambda) \  \colon \ \gamma \in O(M,h)|_{\Lambda}\},$$
by a result of Nikulin~\cite[Corollary 1.5.2]{Nikulin} (Nikulin's result is about the group $\widetilde{O}(\Lambda)$, but nevertheless yields the above when restricting to the subgroup $\widetilde{O}^+(\Lambda)$).

\subsection{Noether-Lefschetz locus}\label{ss:NL} 
We say that $X$ {\it possesses a special labelling of discriminant $d$} if there exists a primitive sublattice $K \subset \Pic X$ of rank 2 with $H \in K$ such that $|D(K^\perp)| = d$. A very general fourfold $X$ in $\calM$ has $\rk \Pic X = 1$ (see~\cite[Section 5.1]{Zarhin} for a standard argument for this fact) and thus does not possess any special labelling. The following result of Debarre, Iliev, and Manivel classifies all possible special labelling (we are able to employ their result because the nonspecial cohomology lattice of a discriminant $d$ Gushel-Mukai fourfold is isomorphic to the nonspecial cohomology lattice of a discriminant $d$ $K3^{[2]}$ fourfold):

\begin{theorem}\cite[Proposition 6.2]{DIM}
\label{thm:orbits}
 A special sublattice $K$, i.e. a rank 2 sublattice $K \subset M$ with $u+v \in K$ of signature $(1,1)$, must have discriminant $d \equiv 0,2,4 \bmod 8$.
 Furthermore, the orbits of $O^+(\Lambda)$ acting on the set of special rank 2 sublattices are as follows:
\medskip \begin{enumerate}
\item If $d =8m$, there is just one orbit for each $m>0$, represented by $K_d$ with $K_d \cong \begin{pmatrix} 2 & 0 \\ 0 & -2m   \end{pmatrix}$ and $K_d \cap \Lambda = \bZ(u' -mv')$. \medskip

\item If $d = 8m+2$, there are two orbits for each $m>0$, exchanged by an automorphism of $\Lambda$ switching $w$ and $u-v$. Both of these orbits consist of lattices isomorphic to  $\begin{pmatrix} 2  & 0 \\ 0 & -2 - 8m\end{pmatrix}$. One of these orbits has representative $K_d'$ with $K_d' \cap \Lambda = \bZ(u-v + 2u' - 2m v')$. The other has representative $K_d''$ such that $K_d'' \cap \Lambda= \bZ(w + 2u' - 2m v')$.
\medskip

\item If $d =8m+4$, there is just one orbit for each $m>0$. This orbit has a representative $K_d$ with $K_d \cong \begin{pmatrix} 2 & 0\\ 0 & -4 - 8m \end{pmatrix}$, and $K_d \cap \Lambda= \bZ(u-v + w + 2u' - 2mv')$.
\end{enumerate}
\end{theorem}
Using~\cite[Corollary 1.5.2]{Nikulin} once again, we observe that
$$\widetilde{O}^+(\Lambda, K_d \cap \Lambda)|_{K_d^\perp} = O^+(M, K_d)|_{K_d^\perp} = \widetilde{O}^+(K_d^\perp)$$
and
\begin{equation} \label{eqn:gammaddef}\Gamma_d \defeq \widetilde{O}^+(\Lambda, (K_d \cap \Lambda))|_{K_d^\perp} = (O^+(M, h) \cap O^+(M, (K_d)))|_{K_d^\perp} = \langle \widetilde{O}^+(K_d^\perp), -\id_{K_d^\perp} \rangle.\end{equation}
In particular, the group $\widetilde{O}^+(K_d^\perp)$ is an index  2 subgroup of $\Gamma_d$.  

We define the divisor $\calD_d \subset \calD$ for each $d \equiv 0,2,4$ as in Theorem~\ref{thm:orbits} as follows:
For $d \equiv 0,4 \bmod 8$, define 
$$\Omega_d^+ \defeq \{ \omega \in  \Omega_\Lambda^+ \ \colon \ \omega^\perp  \supseteq K_d \cap \Lambda \};$$
Then $\calD_d$ is the image of $\Omega_d^+$ under the projection map $\Omega_\Lambda^+ \to \calD_\Lambda$, and is an irreducible divisor. We define $\calM_d$ to be $\calM_d \defeq \tau^{-1}(\calD_d)$; when nonempty, this is a divisor in $\calM$. Note that $\calM_d$ parameterizes the $(X,H) \in \calM$ that possess a special labelling of discriminant $d$. For $d \equiv 2 \bmod 8$, the irreducible divisors $\calD_d', \calD_d'' \subseteq \calD$ and $\calM_d \subseteq \calM$ are similarly defined. 

The following theorem of Debarre and Macr\`{i}, a consequence of \cite[Proposition 4.1 and Theorem 6.1]{DM}, gives the image of $\tau$:

\begin{theorem}[Debarre-Macr\`{i}]
\label{thm:image}
The image of the Torelli map $\tau \colon \calM \to \calD$ meets exactly the following divisors ($d>0$):
\medskip
\begin{enumerate}
\item If $d \equiv 0,4 \bmod 8$, the image meets $\calD_d$ except for $d = 4$ and $d=8$.
\medskip
\item If $ d\equiv  2 \bmod 8$, the image meets $\calD_d'$ and $\calD_d''$, except for: $d=2$, and one of $\calD_d', \calD_d''$ for $d=10$.
\end{enumerate}
\end{theorem}

To prove Theorem~\ref{thm:MainThm}, it suffices to compute the Kodaira dimension for $\calD_d$, since $\calM_d$ and $\calD_d$ are birational. 

\begin{notational}
For $d \equiv 2 \bmod 8$, we will set $\calD_d = \calD_d'$, as we only care about Kodaira dimension, and $\calD_d'$ is isomorphic to $\calD_d''$. We will also set $K_d = K_d'$. 
\end{notational}

\subsection{Orthogonal modular varieties}
\label{ss:orthogmodvars}
Let us now relate $\calD_d$ via a birational map to an {\it orthogonal modular variety}, that is, a quotient of the form $\Gamma \backslash \Omega_L^+$ for any $\Gamma \subseteq O^+(L)$ of finite index. Our approach to finding an appropriate orthogonal modular variety $\calF_d$ birational to $\calD_d$ is inspired by Hassett's work~(\cite{HassettThesis},~\cite{HassettCompositio}) on the analogous problem for special cubic fourfolds, which is lucidly explained in~\cite{Huybrechts} and in~\cite{Brakkee}. Then we discuss how to apply the low-weight cusp form trick.

Recall that $K_d^\perp$ denotes the orthogonal complement (in $M$) of the representative $K_d$ given in Theorem~\ref{thm:MainThm}. We defined~\eqref{eqn:gammaddef} a group $\Gamma_d \subset O^+(K_d^\perp)$ which contains $\widetilde{O}^+(K_d^\perp)$ as an index 2 subgroup. We have natural morphisms of algebraic varieties:
\begin{equation}  \label{eqn:threequotients1} \calG_d \defeq \widetilde{O}^+(K_d^\perp)  \backslash \Omega^+_{K_d^\perp} \to \calF_d \defeq \Gamma_d \backslash \Omega^+_{K_d^\perp}  \to  \widetilde{O}^+(\Lambda) \backslash \Omega^+_{\Lambda} = \calD \end{equation}
By definition, the image of the second morphism in~\eqref{eqn:threequotients1} is $\calD_d$, so we may rewrite these morphisms as
\begin{equation} \label{eqn:threequotients2} \calG_d \xrightarrow{\phi} \calF_d \xrightarrow{\psi} \calD_d. \end{equation}
The variety $\calG_d$ parametrizes {\it marked} special weight 2 Hodge structures of $K3$ type on $K_d^\perp$ (a Hodge structure on $K_d^\perp$ together with the data of a lattice embedding $K_d \hookrightarrow M$) , while $\calF_d$ parametrizes {\it labelled} weight 2 Hodge structures of $K3$ type on $K_d^\perp$ (Hodge structures on $M$ together with the data of the {\it image} of a lattice embedding $K_d \hookrightarrow M$). 
\begin{remark} We note that since $-\id$ acts as the identity on $\Omega^+_{K_d^\perp}$, we have that $\calF_d = \calG_d$. We choose to work with $\calF_d$ to avoid the potential issues due to irregular cusps (although this only happens when $d=32$, see~\cite{MaCusps}), and because the property that $-\id \in \Gamma_d$ will be useful in~\S\ref{ss:gammadmodularity}.
\end{remark}

The next proposition, whose proof we mirror on similar arguments appearing in~\cite[Corollary 2.5]{Huybrechts} and~\cite{Brakkee}, has the key consequence that the morphism $\psi$ appearing in~\eqref{eqn:threequotients2} is generically injective:
\begin{proposition}
\label{prop:threequotients}
The morphism $\psi$ is the normalization of $\calD_d$.
\end{proposition}
\begin{proof}
We show $\psi$ is finite of degree $1$. We begin by showing the properness of $\psi$: start with observation that the morphisms (in the complex analytic category)
$\Omega_{\Lambda}^+ \to \calD_\Lambda$, $\Omega_{K_d^\perp}^+ \to  \Omega_\Lambda^+$, and $\Omega_{K_d^\perp}^+ \to \calF_d$ are {\it closed}, and that the composition $\Omega_{K_d^\perp}^+ \to   \Omega_\Lambda^+ \to  \calD_d$ is closed as well. Since we can further factor this closed morphism into the composition of two other morphisms with the first being closed,
$$\Omega_{K_d^\perp}^+ \to \calF_d \to \calD_d,$$
it follows that $\calF_d \to \calD_d$ is closed. Since each fiber is a compact set --- indeed a finite set--- this is a proper morphism. 
 Furthermore, as $\psi$ is quasi-finite and proper, it follows that $\psi$ is finite. \

Let $n$ denote the degree of $\psi$, i.e. there is an open set $U \subseteq \calF_d$ such that, for any $x \in U$, the fiber $\psi^{-1}(x)$ has cardinality $n$. Since a very general $(X, H) \in \calM_d$ has $\rk(\Pic X) = 2$ (again by the reasoning in~\cite[Section 5.1]{Zarhin}), a very general fiber must consist of a single point. Therefore, we have $n=1$ and so $\psi$ is a birational morphism. By~\cite{BB}, the variety $\calF_d$ is normal, so $\calF_d$ must be the normalization of $\calD_d$. 
\end{proof}
Since $\psi$ is a birational map, we may conclude 
$$\kappa(\calF_d) = \kappa(\calD_d) = \kappa(\calM_d).$$

To use the low-weight cusp-form trick to compute $\kappa(\calF_d)= \kappa(\calM_d$), we review a little theory of modular forms on orthogonal groups.  Let $L$ be a signature $(2,n)$ lattice with $n\ge 3$, let $\Gamma \subseteq O^+(L)$ be a finite index subgroup, let $\chi \colon \Gamma \to \bC^\times$ be a character, and let $\Omega_L^{+\bullet}$ denote the affine cone over $\Omega_L^+$. A {\it modular form of weight $k$ with character $\chi$ for the group $\Gamma$ } is a holomorphic function $F \colon \Omega_L^{+\bullet} \to \bC$ satisfying the following properties for all $z \in \Omega_L^{+\bullet}$:
\begin{enumerate} \medskip
\item For every $\gamma \in \Gamma$, we have $F(\gamma z) = \chi(\gamma) F(z)$
\medskip
\item For every $t \in \bC^\times$, we have $F(tz) =t^{-k} F(z)$. 
\end{enumerate}
Let us denote by $M_k(\Gamma, \chi)$ the collection of all such modular forms.
A {\it cusp form} is a modular form $F \in M_k(\Gamma, \chi)$ vanishing at the cusps of the Baily-Borel compactification of the variety $\Gamma \backslash \Omega_L^+$, and all such forms form a vector space denoted $S_k(\Gamma, \chi)$. The low-weight cusp form trick is summarized in the following theorem of Gritsenko, Hulek, and Sankaran:
\begin{theorem}(\cite[Theorem 1.1]{GHSInventiones} and~\cite{MaCusps})
\label{thm:cuspformtrick}
Let $L$ be a lattice of signature $(2,n)$ with $n \ge 9$ and $\Gamma \subseteq O^+(L)$ a subgroup of finite index containing $-\id$. The variety $\Gamma \backslash \Omega_L^+$ is of general type if there exists a cusp form $F$ for the group $\Gamma$ with weight $a < n$ and character $\chi$ such that $F$ vanishes along the divisor of ramification of the projection map $ \Omega^+_L \to \Gamma \backslash \Omega^+_L$. If there is a nonzero cusp form of weight $n$  for $\Gamma$ with character $\det$, then $\kappa(\Gamma \backslash \Omega^+_L) \ge 0.$
\end{theorem}

To apply Theorem~\ref{thm:cuspformtrick} to compute the Kodaira dimension of $\Gamma_d  \backslash \Omega^+_{K_d^\perp} $, one needs a supply of modular forms which are modular with respect to $\Gamma_d$. For us, these are provided by {\it quasi-pullbacks} of modular forms with respect to some higher rank orthogonal group, which we now describe. Let $L_{2,26}$ denote the unique even unimodular lattice of signature $(2,26)$:
$$L_{2,26} = U^{\oplus 2} \oplus E_8(-1)^{\oplus 3}$$

It is known ~(\cite{Borcherds}) that $M_{12}(O^+(L_{2,26}), \det)$ is a one-dimensional complex vector space spanned by a modular form $\Phi_{12}$, called the {\it Borcherds form}. The divisor of zeros of $\Phi_{12}$ is the union 
\begin{equation} \label{eqn:div} \text{div}(\Phi_{12}) = \bigcup_{r \in L_{2,26}, r^2=-2}  \Omega^+_{L_{2,26}}(r), \end{equation} 
where $\Omega^+_{L_{2,26}}(r)$ denotes a rational quadratic divisor as in~\S\ref{ss:lattices}, and the order of vanishing of $\Phi_{12}$ is exactly 1 along each such divisor.
Given a primitive embedding of lattices $\iota \colon L \hookrightarrow L_{2,26}$, with $L$ of signature $(2,n)$, let
$$R_{-2}(\iota) \defeq \{r \in L_{2,26} \colon r^2=-2, (r,\iota( L)) = 0\}.$$
When the embedding is clear from context, we may sometimes write $R_{-2}(L)$.
To construct a modular form for some subgroup of $O^+(L)$, one might try using the pullback of $\Phi_{12}$ along the naturally induced closed immersion $ \Omega_{L}^{+\bullet}  \to  \Omega_{L_{2,26}}^{+\bullet}$. But for any $r \in R_{-2}(L)$,  one has $\Omega_{L}^{+\bullet} \subset \Omega_L^+(r)$, and hence $\Phi_{12}$ vanishes identically on $\Omega_{L}^{+\bullet}$. The method of the quasi-pullback, due to Gritsenko, Hulek, and Sankaran, deals with this issue by dividing out by appropriate linear factors:
\begin{theorem}\cite[Theorem 8.2]{GHSHandbook}
\label{thm:quasipullback}
Let $L$ be a lattice of signature $(2,n)$, with $3 \le n \le 26$. Given a primitive embedding of lattices $\iota \colon L \hookrightarrow L_{2,26}$ and the naturally induced embedding $ \Omega_{L}^{+\bullet}  \to  \Omega_{L_{2,26}}^{+\bullet}$, the set $R_{-2}(L)$ of $(-2)$-vectors of $L_{2,26}$ orthogonal to $L$ is a finite set. The {\it quasi-pullback} of $\Phi_{12}$ with respect to this embedding
 $$\Phi|_{\iota(L)} \defeq \frac{\Phi_{12}(Z)}{\prod_{r \in R_{-2}(L)/\pm1} (Z, r)} |_{ \Omega_{L}^{+\bullet} }$$
is a nonzero modular form in $M_{N(\iota(L)) + 12}(\widetilde{O}^+(L), \det)$ where $N(\iota(L)) \defeq \# R_{-2}(\iota)/2$. If $N(\iota(L)) > 0$, then $\Phi|_{\iota(L)}$ is a cusp form.
\end{theorem} 
We will need modularity with respect to $\Gamma_d$, so we will need to be careful that our quasi-pullbacks are modular with respect to the additional transformation $-\id$. Throughout this paper, when an underlying embedding $\iota \colon K_d^\perp \hookrightarrow L$ is clear from context, we will adopt the notation $\Phi|_{K_d^\perp}= \Phi|_{\iota}$ and $N(K_d^\perp) = N(\iota)$.

Thus, to show that $\kappa(\calM_d)=19$, we will first construct embeddings $\iota \colon K_d^\perp \hookrightarrow L_{2,26}$ such that $0<N(\iota)<7$, and using the quasi-pullback trick this gives a modular form $\Phi |_{\iota(K_d^\perp)}$ of weight $12+N(K_d^\perp)$ (if an embedding of $K_d^\perp$ satisfies $N(K_d^\perp) = 7$, we may still use this embedding in a proof that $\kappa(\calM_d) \ge 0$). These embeddings will automatically be modular with respect to $\widetilde{O}^+(K_d^\perp)$. Still, there is nothing in Theorem~\ref{thm:quasipullback} to guarantee automatically that $\Phi|_{K_d^\perp}$ vanish along the ramification divisor. We will deal with this in~\S\ref{S:constructingembeddingsgeneralities}, where we see how the extra condition that the quasi-pullback is modular with respect to $\Gamma_d$ guarantees this vanishing.

\section{Constructing embeddings: generalities}\label{S:constructingembeddingsgeneralities}
In this section, we begin constructing embeddings $K_d^\perp \hookrightarrow L_{2,26}$ such that $N(K_d^\perp) <7$. Let us first write down the lattices $K_d^\perp$ we are studying. Using the representatives from Theorem~\ref{thm:orbits}, we compute the lattices $K_d^\perp$. The results of this straightforward computation are summarized in the following proposition. We introduce for ease of notation lattices $M_d$ defined by their Gram matrices (see also~\cite[Proposition 6.2]{DIM} and~\cite[Lemma 4.6]{Pertusi}):
$$d=8m, \ M_d \defeq \begin{pmatrix} -2 & 0 & 0 \\ 0 & -2 & 0 \\ 0 & 0 & 2m \end{pmatrix}$$
$$d=8m+2,  \ M_d \defeq \begin{pmatrix} -2 & 0 & 0 \\ 0 & -2 & 1 \\ 0 & 1 & 2m \end{pmatrix}$$
$$d=8m+4,  \ M_d \defeq  \begin{pmatrix} -2 & 0 & 1 \\ 0 & -2 & 1 \\ 1 & 1 & 2m \end{pmatrix}$$

\begin{proposition}
\label{prop:Kdperps}
Let $K_d$ be the representative rank $2$ lattice from Theorem~\ref{thm:orbits}. Then 
$$K_d^\perp \cong M_d \oplus U \oplus E_8^{\oplus 2}(-1).$$
\end{proposition}

Note that in every $M_d$, there is a primitively embedded copy of the lattice $A_1(-1)^{\oplus 2}$ corresponding to the upper-left $2\times 2$ block in the Gram matrix of $M_d$, so from here on we will refer to a sublattice $A \defeq A_1(-1)^{\oplus 2}  \subset M_d$. 
 
We want to consider as many embeddings $K_d^\perp \hookrightarrow L_{2,26}$ as possible. We will label the factors in our decomposition of $L_{2,26}$ as follows:
$$L_{2,26} = U_1 \oplus U_2 \oplus E_8(-1)^{(1)} \oplus E_8(-1)^{(2)} \oplus E_8(-1)^{(3)}.$$
By Nikulin's analog of Witt's theorem (see~\cite[Theorem 1.14.4]{Nikulin}), a primitive embedding $U \oplus E_8(-1)^{\oplus 2} \hookrightarrow L_{2,26}$ is unique up to isometry of $L_{2,26}$, and the same is true for any primitive embedding $A_1(-1)^{\oplus 2} \hookrightarrow U \oplus E_8(-1)$. Thus, without loss of generality, we will from now on assume that all of our embeddings:
\medskip \begin{enumerate}
\item   identify the factor $U \oplus E_8^{\oplus 2}(-1)$ appearing in our decomposition of $K_d^\perp$ in Proposition~\ref{prop:Kdperps} with $U_1 \oplus E_8(-1)^{(1)}  \oplus E_8(-1)^{(2)} \subset L_{2,26}$ ; and
\medskip
\item  Isometrically embed $A_1(-1)^{\oplus 2} \subset M_d$ into $E_8(-1)^{(3)}$. Let $a_1,a_2$ denote the images of generators of the two $A_1(-1)$ summands.
\end{enumerate}
So the problem of writing down embeddings to prove Theorem~\ref{thm:MainThm} is reduced to choosing $\ell \in U_2 \oplus E_8(-1)^{(3)}$ such that $\ell^2=2m$ and
\begin{equation}
\label{eqn:admissible}
\begin{cases}
			(\ell,a_1)=(\ell,a_2)=0 & \textrm{if } d = 8m, \\
			(\ell, a_1) =1, (\ell, a_2) = 0 & \textrm{if } d = 8m + 2 \\
			(\ell, a_1) = (\ell, a_2) = 1 & \textrm{if } d = 8m + 4.		
\end{cases}
\end{equation}
We will say that a vector $\ell =\alpha e + \beta f + v$, where $\{e,f\}$ is a null basis fo $U_2$, $v \in\ E_8(-1)^{(3)}$, and $\ell^2 = 2m$, is {\it admissible for $d$} if one of the three equations in~\eqref{eqn:admissible} holds. Note that if a vector $\ell$ is admissible, there is a unique associated discriminant $d \in \{8m, 8m+2, 8m+4\}$ such that~\eqref{eqn:admissible} is true. For admissible $\ell$ and its associated discriminant $d$, we introduce the following notations:
\begin{itemize}
\item $\iota_\ell \colon K_d^\perp \hookrightarrow L_{2,26}$ is the embedding associated to $\ell$
\item $R_\ell$ is the set $R_{-2}(\iota_\ell(K_d^\perp))$
\item $N_\ell = \#R_\ell/2$.
\item $\Phi_\ell$ is the modular form $\Phi|_{\iota_\ell(K_d^\perp)}$. 
\end{itemize}

\begin{remark}
Every primitive embedding $K_d^\perp \hookrightarrow L_{2,26}$ is isometric to $\iota_\ell$ for some admissible $\ell$---although not every admissible $\ell$ yields primitive $\iota_\ell$. In what follows, we can guarantee that a choice of $\ell$ gives a primitive embedding $\iota_\ell$  whenever $\alpha$ and $\beta$ are coprime.  
\end{remark}
 
For each $d$, we wish to find admissible $\ell$ such that the following hold:
\medskip \begin{enumerate}[(A)]
\item \label{desiderataA} $\iota_\ell$ is primitive and $0<N_\ell<7$ with $N_\ell$ odd (or $0<N_\ell\le 7$ with $N_\ell$ odd if attempting to prove $\kappa(\calM_d) \ge 0$).\medskip
\item \label{desiderataB} $\Phi_\ell$ vanishes along the ramification locus of the projection $\Omega^+_{K_d^\perp} \to \Gamma \backslash \Omega^+_{K_d^\perp}$.
\end{enumerate}
Then we can apply Theorem~\ref{thm:quasipullback} to these embeddings to produce the cusp forms we need to prove Theorem~\ref{thm:MainThm}. The condition that $N_\ell$ is odd will guarantee that the cusp form vanishes along the ramificiation divisior, as we explain later in this section.

The remainder of the paper will be dedicated to the search for admissible $\ell$ with these desired properties.
 
\subsection{Controlling the size of $R_\ell$} The next two lemmas from~\cite[Section 4]{TVACrelle}, which we state in a slightly more general form, will help us count the number of roots $R_\ell$. Recall one of our goals~((\ref{desiderataA}) above) is to keep $N_\ell$ small.
\begin{lemma}
\label{lem:types}
 Let $L=U \oplus E_8(-1)$ where $U = \langle e, f \rangle$ with $e^2=f^2=0$ and $(e,f)=1$, and let $L_0$ be a primitive rank $2$ sublattice of $E_8(-1)$. Let $\ell \in L$ have length $\ell^2=2m$, for some $m>0$ a positive integer, such that $\ell = \alpha e + \beta f + v$ with $\alpha, \beta \in \bZ$ and $v \in E_8(-1)$, and suppose further that $\alpha \ne  \beta$ and $m<\alpha \beta < 2m.$ Let $R_\ell$ denote the finite set
 $$\{r \in U \oplus (L_0)^\perp_{E_8(-1)}\ \colon \ r^2=-2, (r, \ell) = 0 \}.$$
Let $r =\alpha' e + \beta' f + v' \in R_\ell$. Then  $\alpha'\beta'= 0$ and there are three types of vectors $r \in R_\ell$:
\medskip \begin{enumerate}
\item Type I vectors $r= v'$. In this case $\alpha'=\beta'=0$ and $r \in (L_0)^\perp_{E_8(-1)}$. \medskip
\item Type II vectors $r = \alpha' e + v'$, $\alpha' \ne 0$. In this case, $(v, v') \equiv 0 \bmod \beta$. \medskip
\item Type III vectors $r=\beta' f + v'$, $\beta \ne 0$. In this case, $(v,v') \equiv 0 \bmod \alpha$.
\end{enumerate}
\end{lemma}
\begin{proof} See~\cite[Lemma 4.1]{TVACrelle} and ~\cite[Remark 4.2]{TVACrelle}. The proof there works for this slightly more general statement, as it only relies on the Cauchy-Schwarz inequality and on the negative definiteness of $L_0$. 
\end{proof}

Imposing slightly stronger inequalities, we get an even stronger statement:
\begin{lemma}
\label{lem:typeIlem}
{\cite[Lemma 4.3]{TVACrelle}}
Suppose we are in the situation of Lemma~\ref{lem:types}, and suppose furthermore that the following three inequalities hold:
$$\alpha > \sqrt{m}, \  \beta > \sqrt{m},  \ \alpha\beta < \frac{5m}{4}.$$
Then every $r \in R_\ell$ is a vector of Type I, i.e. $r \in (L_0)^\perp_{E_8(-1)}$.
\end{lemma}
\begin{proof}
Let $r=\alpha' e + \beta' f + v' \in R_\ell$. Since $\alpha' \beta'  =0$ by Lemma~\ref{lem:types}, it follows that $(v')^2=-2$. Then by Cauchy-Schwarz,
$$(v,v') \le \sqrt{2} |v^2| = \sqrt{4(\alpha\beta-n)} < \sqrt{4\left(\frac{5n}{4} - n\right)} = \sqrt{n}.$$
But then $(v,v')$ is not divisible by $\alpha$, nor by $\beta$, by the first two inequalities in the hypotheses above. So $r$ is of Type I.
\end{proof}
\begin{remark} In fact, for our embeddings, we will want to impose a stronger condition for $\alpha$ and $\beta$, for some $\rho >0$ to be determined later:
\begin{equation} \label{eqn:inequality}
 \sqrt{(1+\rho)m} < \alpha < \sqrt{\frac{5m}{4}}, \ \sqrt{(1+\rho)m} < \beta < \sqrt{\frac{5m}{4}}
\end{equation}

\end{remark}

\subsection{Modularity with respect to $\Gamma_d$} \label{ss:gammadmodularity}
The quasi-pullback $\Phi_\ell$ along any one of our embeddings is already modular with respect to $\widetilde{O}^+(K_d^\perp)$. Since $\calF_d =\calG_d$, we could simply work with the smaller modular group $\widetilde{O}^+(K_d^\perp) \subset \Gamma_d$ and then verify that $\Phi_\ell$ vanishes along the ramification divisor along the lines of~\cite[Proposition 8.13]{GHSHandbook}. We offer an alternative approach to the vanishing along the ramification divisor using modularity with respect to the larger group $\Gamma_d$. 
\begin{remark}
Our results are unchanged whether we consider modularity with respect to $\widetilde{O}^+(K_d^\perp)$ or $\Gamma_d$, since all modular forms for $\widetilde{O}^+(K_d^\perp)$ computed in the small discriminant range in~\S\ref{S:constructingembeddingsspecifics} are also modular with respect to $\Gamma_d$.
\end{remark}
 We would like to choose $\ell$ such that
 $\Phi|_\ell$ is in addition modular with respect to $-\id \in O(K_d^\perp)$. Then  $ \Phi|_{K_d^\perp}$ will be modular with respect to $\Gamma_d$ since $-\id$ and $\widetilde{O}(K_d^\perp)^+$ generate $\Gamma_d$. But since we already know that $\Phi|_{K_d^\perp}$ is $\widetilde{O}^+(K_d^\perp)$-modular by~\ref{thm:quasipullback}, then we can conclude that 
 $$\Phi_\ell(-\id Z) = \Phi_\ell(-Z) = (-1)^{N_\ell}\Phi_\ell(Z).$$
 As a consequence, we have shown the following important lemma:
 \begin{lemma}
 \label{lem:odd}
Let $\iota\colon L \hookrightarrow L_{2,26}$ be a primitive embedding of lattices as in Theorem~\ref{thm:quasipullback} . Then $\Phi|_L$ is modular with respect to $-\id \in O^+(L^\perp)$ if and only if $N(\iota(L))$ (as defined in Theorem~\ref{thm:quasipullback}) is odd.
\end{lemma}
Thus, to guarantee $\Gamma_d$-modularity of the quasi-pullback, we want to be certain that each embedding $\iota_\ell$ which we construct has the property that $N(\iota_\ell(K_d^\perp))$ is odd (this is why we said as much in~\ref{desiderataA}).

The main purpose for us in asking for modularity with respect to $\Gamma_d$ is guaranteeing vanishing along the ramification divisor, which we explain now. For $r \in L$ such that $r^2<0$, we say that $r$ is {\it reflective} whenever the reflection 
$$\sigma_r \colon v \mapsto v- 2\frac{(v,r)}{(r,r)}r$$
is an isometry of $L$, i.e. $\sigma_r \in O(L)$. A rational quadratic divisor $\Omega_L^+(r)$ is said to be a {\it reflective divisor} if $r$ is reflective. The following proposition of Gritsenko, Hulek, and Sankaran describes the ramification divisor of the projection $\Omega^+_L \to \Gamma \backslash \Omega^+_L$ as a union of certain reflective divisors: 

\begin{proposition}(see~\cite[Corollary 2.13]{GHSInventiones}) Let $L$ be a lattice of signature $(2,n)$ and $\Gamma$ be a finite index subgroup of $O^+(L)$. Then the ramification divisor $\text{Bdiv}(\pi_\Gamma)$ of the projection $\pi_\Gamma \colon \Omega^+_L \to \Gamma \backslash \Omega^+_L$ is given as the countable union
$$\text{Bdiv}(\pi_\Gamma) = \bigcup_{\substack{r \in L \text{ primitive }\\ r^2<0 \\ \pm \sigma_r \in \Gamma}} \Omega_L^+(r).$$
\end{proposition}

Let us now apply the above proposition to a modular form $\Phi \in M_k(\Gamma_d, \det)$. We first observe that $-\sigma_r \in \Gamma_d \iff \sigma_r \in \Gamma_d$. Thus, to prove $\Phi$ vanishes along $\text{Bdiv}(\pi_{\Gamma_d})$, it suffices to show that $\Phi$ vanishes on all reflective divisors $\Omega_{K_d^\perp}^+(r)$ with $\sigma_r \in \Gamma_d$. By modularity,  we have $\det(\sigma_r)\Phi(Z) =\Phi(\sigma_rZ)$ for all $Z \in \Omega_{K_d^\perp}^{+,\bullet}$. We observe that $\det(\sigma_r) = -1$ and $(\sigma_r)|_{\Omega_{K_d^\perp}^+(r)^\bullet}= \id$. It follows that $\Phi$ vanishes on $\Omega_{K_d^\perp}^+(r)^\bullet$. This yields the following proposition: 

\begin{proposition}
\label{prop:ram}
Every modular form for $\Gamma_d$ with character $\det$ vanishes along the ramification divisor.
\end{proposition}

\section{Constructing embeddings: specifics}
\label{S:constructingembeddingsspecifics}
In this section, we prove Theorem~\ref{thm:MainThm}. This will follow from the following proposition:

  \begin{proposition}
\label{prop:mainprop}
For each discriminant for which we claim $\calM_d$ is of general type in Theorem~\ref{thm:MainThm}, there is some $\ell$, \textit{admissible for $d$}, which satisfies conditions~(\ref{desiderataA}) and~(\ref{desiderataB}) above. \end{proposition}

\begin{proof}[Proof of Theorem~\ref{thm:MainThm}, assuming Proposition~\ref{prop:mainprop}]
For each $d$ in the theorem statement, there is some $\ell$ from Proposition~\ref{prop:mainprop} such that the quasi-pullback (Theorem~\ref{thm:quasipullback}) $\Phi_\ell$ is a nonzero cusp form of weight $\le 19$ for $\Gamma_d$ with character $\det$. Furthermore, by Proposition~\ref{prop:ram}, this quasi-pullback vanishes along the ramification locus of $\Omega^+_{K_d^\perp} \to \Gamma \backslash \Omega^+_{K_d^\perp}$. It follows from~\ref{thm:cuspformtrick} that $\Gamma \backslash \Omega^+_{K_d^\perp}$ is a variety of general type. 
\end{proof}

All that remains to do is provide a proof for Proposition~\ref{prop:mainprop}. The rest of the paper is dedicated to this goal.

Given an embedding $\iota_\ell$, we may count $N_\ell$ with the help of Lemmas~\ref{lem:types} and~\ref{lem:typeIlem} using $L_0=\langle a_1, a_2\rangle $, in which case $(L_0)^\perp_{E_8(-1)} = D_6(-1)$. The upshot of Lemma~\ref{lem:typeIlem} is that, for any admissible $\ell = \alpha e + \beta f + v $ such that $\alpha$ and $\beta$ satisfy the inequalities~\eqref{eqn:inequality}, the set $R_\ell$ is contained entirely in $D_6(-1)$:
$$R _\ell = \{r \in D_6(-1) \  \colon \ r^2 = -2, (r, \ell) = 0\}.$$

\begin{proof}[Proof of Proposition~\ref{prop:mainprop}]
We need to construct primitive embeddings $\iota_\ell$ associated to $\ell= \alpha e + \beta f + v$ such that $0<N_\ell \le 7$ and $N_\ell$ is odd. We construct such an $\ell$ for all large $m$ by picking $\alpha, \beta$ such that~\eqref{eqn:inequality} holds, and can pick $v$ thanks to Lemma~\ref{lem:threesquares} below. We then compute a lower bound on the discriminants for which these conditions can always be met. This leaves us with a finite list of discriminants to analyze. We handle these cases with a computer, giving a summary of this procedure in~\S\ref{ss:computersearch}. We break our analysis into the three cases of discriminant congruent to $0,2$, or $4$ modulo $8$ in sections~\S\ref{ss:d8m},~\S\ref{ss:d8m+2}, and~\S\ref{ss:d8m+4}.

\subsection{Analysis: $d=8m$}
\label{ss:d8m}
For the case $d=8m$, we are searching for $\alpha, \beta$, and $v$ such that $\ell=\alpha e + \beta f + v$ of length $2m$ is admissible for $d=8m$. For the admissibility of $\ell$, it is necessary and sufficient that $(\ell, a_1)=(\ell, a_2)=0$ (by~\eqref{eqn:admissible}), which amounts to requiring $v \in D_6(-1)$. The next lemma  gives a way to construct $\ell$ such that the associated embedding has small $N_\ell$:
\begin{lemma}
\label{lem:keylemma8m}
 Let $\ell = \alpha e + \beta f + v \in U \oplus D_6(-1)$. Suppose that $\alpha, \beta$ satisfy the inequalities \eqref{eqn:inequality}, and that $v$ is of the form 
\begin{equation} \label{eqn:v8m} v=x_1e_1+ x_2e_2 + x_3 e_3 + e_4 + e_5 \end{equation}
with $x_1,x_2, x_3$ all nonnegative integers, not all equal. Then $N_\ell \le 5$. In particular, $N_\ell$ is always odd in this case, and:
\medskip \begin{enumerate} 
\item  If the nonnegative integers $x_1,x_2, x_3$ are distinct and none of them equal to $1$, then $R_\ell =  \{\pm(e_4-e_5)\}$.  \medskip

\item If the nonnegative integers $x_1, x_2, x_3$ are distinct with $x_j =1$, then $R_\ell =  \{\pm (e_4-e_5), \pm(e_4 - e_j), \pm(e_5 - e_j) \}$.

\item If the nonnegative integers $x_1, x_2, x_3$ are distinct with $x_i=0$ and none of them is equal to 1, then $R_\ell =  \{\pm (e_4-e_5), \pm(e_1 - e_6), \pm(e_1 + e_6) \}$.

\item  If the nonnegative integers $x_1,x_2, x_3$ are distinct with $x_i=0, x_j=1$, then $R_\ell =  \{\pm (e_i +e_6), \pm (e_i-e_6), \pm \pm(e_4 - e_j), \pm(e_5 - e_j) , \pm(e_4-e_5)\}$.  \medskip


\end{enumerate}
\end{lemma}

\begin{proof}
By hypothesis, all vectors in $R_\ell$ are of Type I (Lemma~\ref{lem:types}). We shall write $x_4= x_5=1$ and $x_6=0$. The roots of $D_6(-1)$ are $\pm e_i \pm e_j$, $1 \le i,j \le 6$, $i \ne j$. We have for all such roots $r \in D_6(-1)$,
$(r,v) = \pm (e_i \pm e_j,v) = \pm (x_i \pm x_j)$. . The other cases are proved similarly.
\end{proof} 

Thus, to find $v$ as in the lemma,  it would suffice to pick $\alpha, \beta$ satisfying~\eqref{eqn:inequality} such that $2(\alpha\beta-m-1)$ is a sum of three distinct coprime squares: any triple of {\it distinct} nonnegative integers $(x_1,x_2,x_3)\in \bZ^3_{\ge 0}$ with $\gcd(x_1, x_2, x_3)=1$ which is a solution to
\begin{equation}\label{eqn:lagrange8m} x_1^2 +x_2^2+x_3^2 = 2(\alpha\beta-m-1), \ x_1x_2x_3 \ne 0\end{equation}
yields $v$ for which Lemma~\ref{lem:keylemma8m} applies.
 The next lemma, guarantees the existence of these solutions in many cases, is from~\cite[Section 1, Korollar 1]{HalterKoch}:
\begin{lemma}\label{lem:threesquares}
Every integer $\Delta \not \equiv 0, 4,7 \bmod 8$ with 
$$\Delta \notin \{ 1,2,3,6,9,11,18,19,22,27,33,43,51,57,67,99,102,123,163,177,187,267,627 \} \cup \{N\}$$
 may be written as the sum of three distinct, coprime squares. If the generalized Riemann hypothesis is true for all global $L$-functions, then we may take $N=1$, but if a generalized Riemann hypothesis (GRH) is false for certain $L$-functions, then $N >5\cdot 10^{10}$.
\end{lemma}

We also have the following lemma to give us more flexibility in our choice of $\alpha$ and $\beta$ beyond $(\alpha, \beta) = 1$

\begin{lemma} \label{lem:primlemma} Assume that  $\ell = \alpha e + \beta f + v \in U \oplus E_8(-1)$ has square length $\ell^2 = 2m$, with $v$ primitive in $D_6(-1) = \langle a_1, a_2\rangle^\perp_{E_8(-1)}$ , and furthermore assume that $2 \nmid (\alpha, \beta)$. Then the embedding $\iota_\ell \colon K_{8m}^\perp \hookrightarrow L_{2,26}$ is primitive. \end{lemma}
\begin{proof} 
It is enough to check that $M_d=A_1(-1)^{\oplus 2} \oplus \langle 2m \rangle$ embeds primitively into $U \oplus E_8(-1)$. To show an embedding is primitive, it suffices to show the image of every primitive vector is primitive. Thus, we check that $x u + y \ell $ is primitive in $U \oplus E_8(-1)$ for any relatively prime integers $x$ and $y$ and any primitive vector $u \in \langle a_1, a_2 \rangle$. Suppose that there is a positive integer $n$ dividing $xu +y \ell$ in $U \oplus E_8(-1)$. Then $n | y(\alpha, \beta)$. As $E_8(-1)/(A_1(-1)^{\oplus 2} \oplus D_6(-1)) \simeq \bZ /2 \times \bZ /2$, we must have $n| 2$. It follows that $n | y$, so $n | x$ as well (as $A_1(-1)^{\oplus 2}$ is primitively embedded in $E_8(-1)$). As $x$ and $y$ are coprime, we must have $n=1$, so $xu + y\ell$ is indeed primitive under the embedding $\iota_\ell$, and we conclude that $\iota_\ell$ is primitive.
\end{proof}

To build our desired embeddings, we will show that for $m$ large enough, we can choose $\alpha$, $\beta$ so that: (a) $2 \nmid (\alpha, \beta)$, (b) the inequalities~\eqref{eqn:inequality} hold, and (c) $2(\alpha\beta - m -1)$ is a sum of three distinct coprime nonnegative squares.  Observe that it is necessary and sufficient for (c) to hold that $\alpha \beta - m-1$ be both odd and avoid some finite set of exceptional values (see Lemma~\ref{lem:threesquares}).
Then by Lemma~\ref{lem:primlemma} and Lemma~\ref{lem:keylemma8m}, we get a primitive embedding $\iota_\ell \colon K_d^\perp \to L_{2,26}$ with $N_\ell \in \{1,3\}$.  

We begin by choosing some real number $\rho>0$ such that 
\begin{equation} \label{eqn:alphafreedom} \sqrt{\frac{5m}{4}} - \sqrt{(1+\rho)m} > 2. \end{equation}
If $m \equiv 0 \bmod 2$, we are able to pick $\alpha$ and $\beta=\alpha+1$ satisfying~\eqref{eqn:inequality}, thanks to~\eqref{eqn:alphafreedom}. If $m \equiv 1 \bmod 2$, we again can use~\eqref{eqn:alphafreedom} to pick $\alpha \equiv 1 \bmod 2$ satisfying the inequality for $\alpha$ in~\eqref{eqn:inequality}, and set $\alpha = \beta$. So in any case, with these choices for $\alpha$ and $\beta$, (a), (b) hold, and also the quantity $\alpha\beta-m-1$ is odd.

We also need to ensure that $\alpha\beta-m-1$ misses a finite set of exceptional values. For this, note that \begin{equation}\label{eqn:alphaepislon} \alpha^2 + \alpha - m -1  > \alpha^2 - m -1 > \rho m - 1
\end{equation}
holds for all $m, \alpha$ for which~\eqref{eqn:inequality} holds. So given our choices of $\alpha$ and $\beta$ from the previous paragraph, we have the inequality
$$2(\alpha\beta - m-1) > 2\rho m-2.$$
Now, we impose the additional constraint that 
\begin{equation}
\label{eqn:epsilon}
\rho m >52
\end{equation} 
 guaranteeing that $2(\alpha \beta -m -1)>102$ and thereby avoiding the exceptional values of Lemma~\ref{lem:threesquares} (note that there are no odd values in the list of exceptional values which lie between $103$ and $627$), except perhaps $N$. If $2(\alpha \beta-m-1) = N$, then the inequalities 
$$\beta^2 - \beta -m-1 < \beta^2-m-1<\beta^2-m< \frac{m}{4}  $$
hold under our continuing assumption of~\eqref{eqn:inequality} , so 
$$N<\frac{m}{2}.$$
Therefore, we have $m>10 \cdot 10 ^{10}$. If we take $\rho$ to be sufficiently small and $m$ is large enough, then 
\begin{equation}
\label{eqn:GRHfix}
\sqrt{\frac{5m}{4}} - \sqrt{(1+\rho)m} > 4
\end{equation}
so we can adjust $\alpha$ by $\pm 2$ to avoid $N$ (and still keep the quantity $\alpha \beta-m-1$ odd). 

At this point, we have demonstrated that whenever $m$ and $\rho$ satisfy the inequalities~\eqref{eqn:epsilon} and~\eqref{eqn:alphafreedom}, it is possible to pick $\alpha$ and $\beta$ and $v$ to prove $\calM_d$ is of general type. 
A simple optimization for~\eqref{eqn:alphafreedom} and~\eqref{eqn:epsilon} yields $m \ge 648$ for $\rho = 0.0804$. If $m > 10 \cdot 10^{10}$, then~\eqref{eqn:GRHfix} holds, so $\alpha$ may be adjusted to avoid $N$ if necessary. 

Putting everything together, we have now shown that when $m \ge 648$, Proposition~\ref{prop:mainprop} is true for $d=8m$. 
For the discriminants $d=8m$ with $m <648$, we make use of a computer to find explicit embeddings. See~\S\ref{ss:computersearch} for details.

\subsection{Analysis: $d=8m+2$}
\label{ss:d8m+2}
As in the $d=8m$ case, we are searching for $\alpha, \beta$, and $v$ such that  the square-length $2m$ vector $\ell=\alpha e + \beta f + v$ is admissible (i.e. satisfies~\eqref{eqn:admissible} for $d=8m+2$ and yields a small, odd value for $N_\ell$. For the admissibility of $\ell$, it is necessary and sufficient that the vector $v \in E_8(-1)$ may be written as 
$$v=\frac{-a_2}{2} + v' \in (\langle a_1, a_2 \rangle \oplus D_6(-1))^\vee = \langle a_1, a_2 \rangle^\vee \oplus D_6(-1)^\vee,$$
 where $v' \in D_6(-1)^\vee = (\langle a_1, a_2 \rangle^\perp)^\vee$. 

For each $m$ greater than the lower bound that is to be determined, our argument is written in a way that relies on the choice of $a_1, a_2 \in E_8(-1)$; precisely, for each $m$, we will construct $E_8(-1)$ as a specific overlattice of $A_1(-1)^{\oplus 2} \oplus D_6(-1)$, and then consider embeddings for which $a_1, a_2$ generate image of the summand $A_1(-1)^{\oplus 2}$. The theory of overlattices is explained in~\cite[Section 1.4]{Nikulin}, a consequence of which is the following: there are exactly two unimodular negative definite even integral sublattices $L_1$ and $L_2$ of rank 8 (necessarily isomorphic to $E_8$) contained in $(A_1(-1)^{\oplus 2})^\vee \oplus D_6(-1)^\vee$, each of which corresponds to one of the two maximal isotropic subgroups $L_1/(A_1(-1)^{\oplus 2}) \oplus D_6(-1))$ and $L_2/(A_1(-1)^{\oplus 2} \oplus D_6(-1))$ of $D(A_1(-1)^{\oplus 2} \oplus D_6(-1))$. To describe $L_1$ and $L_2$, let $h_1, h_2$ each denote a generator of an orthogonal summand of $A_1(-1)^{\oplus 2}$, and define elements $b_1, b_{2,p}$ in $\langle h_1, h_2 \rangle^\vee \oplus D_6(-1)^\vee$ by
\[ b_1 \defeq e_1 + \frac{h_1+h_2}{2}\]
\[ b_{2,p} \defeq \frac{1}{2}(e_1+ e_2+e_3+e_4+e_5+e_6 ) + \frac{h_p}{2} \]
where the index $p$ is either $1$ or $2$.
Then $L_p$ is generated as a submodule of $(A_1(-1)^{\oplus 2})^\vee \oplus D_6(-1)^\vee$ by $b_1, b_{2,p}$, and $\langle h_1, h_2 \rangle \oplus D_6(-1)$. 

We now prove two simple lemmas: one will help ensure our eventual choice for $v'$ actually gives an embedding, and the other controls the size of $N_\ell$.
\begin{lemma}
\label{lem:ambientLp8m+2}
 Suppose that $v' \in D_6(-1) \otimes \bQ$ is of the form 

\begin{equation} \label{eqn:v8m+2} v'=\frac{1}{2} (x_1e_1 + x_2e_2 + x_3e_3 + 3e_4 + 3 e_5  +3e_6)\end{equation}

 with $x_i \in \bZ$ all nonnegative \textit{odd}. Then $v' \in D_6(-1)^\vee$ and there is always a choice of $p \in \{1,2\}$ such that $v \defeq v'-\frac{h_2}{2} \in L_p$. 

 \end{lemma}
 \begin{proof} We have $v' \in D_6(-1)^\vee$ because all the coefficients with respect to the $\{e_1, \ldots, e_6\}$ basis are half-integers. For the other statement, we compute
 \[(v,b_1) = \frac{-3+1}{2}=-\frac{x_1}{2} + \frac{1}{2}\]
 \[(v,b_{2,p})=-\frac{1}{4} (9+x_1+x_2+x_3)-\frac{1}{4}(h_2,h_p).\]
These inner products are integer-valued if and only if $v \in L_p^\vee = L_p$. By taking $p=1$ when $x_1+x_2 + x_3 \equiv 3 \bmod 4$ or choosing $p=2$ otherwise, we see there is always $p$ such that $v \in L_p$.
 \end{proof}
 
 \begin{lemma}  \label{lem:keylemma8m+2}
Suppose that 
\begin{itemize}
\item$\alpha, \beta$, and $m$ are positive integers satisfying the inequalities~\eqref{eqn:inequality},
\item $v' = \frac{1}{2} (x_1e_1 + x_2e_2 + x_3e_3 + 3e_4 + 3 e_5  +3e_6) \in D_6(-1)^\vee$, as in Lemma~\ref{lem:ambientLp8m+2},
\item $(v')^2=2(m-\alpha\beta)+ \frac{1}{2}$,
\item the integers $x_1, x_2, x_3$ in $v'$ are distinct integers, none of which are equal to 3.
\end{itemize}
Choose $p \in \{1,2\}$ so that $v'-\frac{h_2}{2}\in L_p$, and fix an identification of $L_p$ with $E_8(-1)$. Let $\iota_\ell$ be the embedding defined by $a_1=h_1, a_2 = h_2,$ and $\ell = \alpha e + \beta f +v' - \frac{a_2}{2}$. Then $R_\ell = \{\pm (e_4-e_5), \pm (e_5-e_6) , \pm (e_4-e_6)\}$.
\end{lemma}

\begin{proof}
Omitted, as it is completely similar to the proof of Lemma~\ref{lem:keylemma8m}.
\end{proof}

Assuming we have chosen $\alpha, \beta$, and $m$ satisfying the inequalities~\eqref{eqn:inequality}, we show that it is always possible to pick $v'\in D_6(-1)$ satisfying the hypothesis of the lemma. A vector $v'$ as in~\eqref{eqn:v8m+2} satisfies
$$-(2v')^2 = x_1^2+x_2^2 +x_3^2 + 27 =-( 8(m-\alpha \beta ) + 2)= 8(\alpha\beta-m) - 2.$$
So it suffices to find a solution to 
\begin{equation} \label{eqn:threesquares8m+2} x_1^2+x_2^2+x_3^2 = 8(\alpha \beta - m) -29 \end{equation}
subject to certain conditions; precisely, we want {\it distinct} nonnegative integer solutions $(x_1,x_2,x_3)$, such that $3 \notin \{x_1, x_2, x_3\}$. Since every square is $0$ or $1 \bmod 4$, it follows that any solution satisfying these conditions is a triple of \textit{odd} integers. As $8(\alpha\beta-m)-29 \equiv 3 \bmod 8$, we can apply Lemma~\ref{lem:threesquares} to find a coprime triple of distinct nonnegative integers $(x_1,x_2,x_3)$ satisfying ~\eqref{eqn:threesquares8m+2} as long as the expression $8(\alpha\beta-m)-29$ avoids a finite list of exceptional values. Suppose that we arrange, by appropriately choosing $\alpha$ and $\beta$, that $3 | 8(\alpha\beta-m)-29$. If $x_i^2 \equiv 0 \mod 3$ for all $i =1,2,3$, then the $x_i$ are not coprime, so we must have $x_i^2 \equiv 1 \mod 3$ for all $i$; in particular, the $x_i$ are distinct from $3$.
 Therefore, if we impose the additional condition on $\alpha, \beta$, and $m$ that $3 | 8(\alpha\beta-m)-29$, then there exists a $v'$ satisfying the hypotheses of Lemma~\ref{lem:keylemma8m+2}.

To build our embeddings, it suffices to arrange that: (a) $(\alpha, \beta)=1$ (to guarantee primitivity), (b) the inequalities~\eqref{eqn:inequality} hold, and (c) $8(\alpha \beta - m) -29$ is a sum of three distinct coprime nonnegative squares.  We have already seen that (c) holds if 
\begin{equation}\label{eqn:avoid8m+2}  8(\alpha\beta-m)-29 >627, \ \ 8(\alpha\beta-m)-29 \ne N \end{equation}
and
\begin{equation}\label{eqn:3divis} 3 | 8(\alpha\beta-m) -29. \end{equation}
If the inequality
\begin{equation} \label{eqn:alphafreedom2} \sqrt{\frac{5m}{4}} - \sqrt{(1+\rho)m} > 6 \end{equation}
holds, then there must exist relatively prime $\alpha, \beta$ satisfying~\eqref{eqn:inequality} such that both $\beta = \alpha+g$ for some $g\in \{1,3\}$ {\it and} $3 | 8(\alpha\beta-m) -29$.

By considering the conditions~\eqref{eqn:alphafreedom2} and~\eqref{eqn:avoid8m+2}, we can now successfully determine a lower bound $m_0$ such that $\calM_{8m+2}$ is of general type for $m \ge m_0$. 
First, note that for $\alpha$, $\beta = \alpha+g$, and $m$ satisfying~\eqref{eqn:inequality}, we have the inequality
\begin{equation}\label{eqn:alphaepislon2} \alpha\beta - m = \alpha^2 + g\alpha - m > \alpha^2 - m > \rho m
\end{equation}
and, as an immediate consequence,
$$8(\alpha\beta - m) -29 > 8\rho m-29.$$
Thus, taking 
\begin{equation}
\label{eqn:epsilon2}
\rho m>82
\end{equation} 
will ensure that $8(\alpha \beta -m)-29>627$.  If $8(\alpha \beta -m)-29= N$, where $N$ is as defined in Lemma~\ref{lem:threesquares}, then the inequalities 
$$\alpha \beta - m = \beta^2 - g\beta -m <\beta^2-m< \frac{m}{4}  $$
hold under our continuing assumptions on $\alpha, \beta = \alpha+g$, and $m$. Therefore for such $N$ we must have
$$N <2m - 29.$$
So we would like to ensure that for $m >(N+29)/2$, the quantity $\rho>0$ is small enough so that the difference 
\begin{equation} \label{eqn:GRHfix2}\sqrt{\frac{5m}{4}} - \sqrt{(1+\rho)m} \end{equation}
is large enough to adjust $\alpha, \beta$ by $\pm 3$ (to preserve~\eqref{eqn:3divis}) in order to avoid $N$.

As before, optimization for~\eqref{eqn:alphafreedom2}) and~\eqref{eqn:epsilon2} yields $m \ge 3238$ and $\rho = 0.025328$. In the range $m \ge 3238$ for this $\rho$, one checks that 
\begin{equation}
\sqrt{\frac{5m}{4}} - \sqrt{(1+\rho)m} > 16000
\end{equation}
so we are always able to adjust $\alpha$ to avoid $N$. As in~\S\ref{ss:d8m}, we now have proven Proposition~\ref{prop:mainprop} is true when 
$m \ge 3238$.
The remaining cases for $d=8m+2$ are handled by computer (see~\S\ref{ss:computersearch}).

\subsection{Analysis: $d=8m+4$}
\label{ss:d8m+4}
Our argument for $d=8m+4$ is nearly identical to the case for $d=8m+2$, but we write out the details since there is a slight variation in the construction we use to produce an explicit lower bound. To precisely state the problem, we wish to show that for all but finitely many positive integers $m$, there are positive integers $\alpha, \beta$, and $v \in U \oplus E_8(-1)$ such that the square-length $2m$ vector $\ell=\alpha e + \beta f + v$ is admissible for $d=8m+4$ and yields a small, odd value for $N_\ell$. For the admissibility of $\ell$, it is necessary and sufficient that the vector $v \in E_8(-1)$ may be written as 
$$v=\frac{-a_1-a_2}{2} + v' \in (\langle a_1, a_2 \rangle \oplus D_6(-1))^\vee = \langle a_1, a_2 \rangle^\vee \oplus D_6(-1)^\vee,$$
 where $v' \in D_6(-1)^\vee=(\langle a_1, a_2 \rangle^\perp)^\vee$.

The following two lemmas adapt Lemmas~\ref{lem:ambientLp8m+2} and~\ref{lem:keylemma8m+2} to the case of $8m+4$. Recall the vectors $h_1, h_2$ are an orthogonal basis for $A_1(-1)$ and $b_1, b_{2,p}$ for $p \in \{1,2\}$ are vectors in $A_1(-1) ^{\vee \oplus 2} \oplus D_6(-1)^\vee$. 
\begin{lemma}
\label{lem:ambientLp8m+4}
 Suppose that $v' \in D_6(-1) \otimes \bQ$ is of the form 

\begin{equation} \label{eqn:v8m+4} v'=x_1e_1 + x_2e_2 + x_3e_3 + 3e_4 + 3 e_5  + 3e_6 \end{equation}

 with $x_i \in \bZ$ all nonnegative integers such that $\sum x_i \equiv 0 \bmod 2$ . Then $v' \in D_6(-1)^\vee$; furthermore, for any isometrically embedded sublattice $A_1(-1)^{\oplus 2} \oplus D_6(-1) \hookrightarrow E_8(-1)$, the image $v$ of $v'-\frac{h_1+h_2}{2}$ under the induced map $(A_1(-1)^{\oplus 2} \oplus D_6(-1))\otimes \bQ  \hookrightarrow E_8(-1)\otimes \bQ $ is an element of $E_8(-1)$. 

 \end{lemma}
 \begin{proof} We have $v' \in D_6(-1)^\vee$ because all the coefficients with respect to the $\{e_1, \ldots, e_6\}$ basis are integers. For the other statement, we recall that for some $p \in \{1,2\}$, $E_8(-1)$ is formed by the span of the isometric image of $\langle h_1, h_2 \rangle \oplus D_6(-1)$ and $b_1, b_{2, p}$. We compute:
 \[(v,b_1) = -x_1+1\]
 \[(v,b_{2,p})=-\frac{1}{2} (9+x_1+x_2+x_3)-\frac{1}{2}.\]
By hypothesis, the right-hand sides of these equalities are integers, and therefore $v \in E_8(-1)$. \end{proof}
 
 \begin{lemma}  \label{lem:keylemma8m+4}
Suppose that 
\begin{itemize}
\item$\alpha, \beta$, and $m$ are positive integers satisfying the inequalities~\eqref{eqn:inequality},
\item $v' = x_1e_1 + x_2e_2 + x_3e_3 + 3e_4 + 3 e_5  +3e_6 \in D_6(-1)^\vee$, as in Lemma~\ref{lem:ambientLp8m+4},
\item $(v')^2=2(m-\alpha\beta) +1$,
\item the integers $x_1, x_2, x_3$ in $v'$ are distinct integers, none of which are equal to 3.
\end{itemize}
Pick any $a_1, a_2$ orthogonal $(-2)$-roots of $E_8(-1)$, and let $\iota_\ell$ be the embedding defined by $a_1=h_1, a_2 = h_2,$ and $\ell = \alpha e + \beta f +v' - \frac{a_1+a_2}{2}$. Then $R_\ell = \{\pm (e_4-e_5), \pm (e_5-e_6) , \pm (e_4-e_6)\}$.
\end{lemma}

\begin{proof}
Omitted, as it is completely similar to the proof of Lemma~\ref{lem:keylemma8m}.
\end{proof}

Assuming we have chosen $\alpha, \beta$, and $m$ satisfying the inequalities~\eqref{eqn:inequality}, we show that it is always possible to pick $v'\in D_6(-1)$ satisfying the hypothesis of Lemma~\ref{lem:keylemma8m+4}. A vector $v'$ as in~\eqref{eqn:v8m+4} satisfies
$$-(v')^2 = x_1^2+x_2^2+ x_3^2 + 27 =2(\alpha\beta-m) - 1$$
So it suffices to find a solution to 
\begin{equation} \label{eqn:threesquares8m+4} x_1^2+x_2^2+x_3^2 = 2(\alpha \beta - m) -28 \end{equation}
subject to certain conditions; precisely, we want {\it distinct}, nonnegative,  coprime integer solutions $(x_1,x_2,x_3)$, such that $3 \notin \{x_1, x_2, x_3\}$. Suppose we have arranged that $2(\alpha \beta - m) - 28\equiv 2 \bmod 4$, or, equivalently, that $\alpha\beta -m$ is odd.  Then we can always solve~\eqref{eqn:threesquares8m+4} (by Lemma~\ref{lem:threesquares}), away from the finite list of exceptional values. Suppose that we have additionally arranged, by appropriately choosing $\alpha$ and $\beta$, that $3 | 2(\alpha\beta-m)-28$. Then each of the integers $x_1, x_2,x_3$ coming from a solution to~\eqref{eqn:threesquares8m+4} must be distinct from 3, or else we would have $3| \GCD(x_1,x_2,x_3)$ (recall we are asking that the $x_i$ are coprime).
 Therefore, if we impose the additional conditions on $\alpha, \beta$, and $m$ that $3 | 2(\alpha\beta-m)-28$ and that $\alpha \beta -m$ is odd, then there exists a $v'$ satisfying the hypotheses of Lemma~\ref{lem:keylemma8m+4}.

To build our embeddings, it suffices to arrange that: (a) $(\alpha, \beta)=1$ (to guarantee primitivity), (b) the inequalities~\eqref{eqn:inequality} hold, and (c) $2(\alpha \beta - m) -28$ is a sum of three distinct coprime nonnegative squares.  We have already seen that (c) holds if 
\begin{equation}\label{eqn:avoid8m+4}  2(\alpha\beta-m)-28 >102,\  2(\alpha\beta-m)-28 \ne N ,\end{equation}
\begin{equation}\label{eqn:3divis4} 3 | 2(\alpha\beta-m) -28, \end{equation}
\begin{equation} \label{eqn:1mod2} \alpha \beta -m \equiv 1 \bmod 2. \end{equation}
If we insist that the inequality
\begin{equation} \label{eqn:alphafreedom4} \sqrt{\frac{5m}{4}} - \sqrt{(1+\rho)m} > 12 \end{equation}
holds, then there must exist relatively prime $\alpha, \beta$ satisfying~\eqref{eqn:inequality} such that (c) holds: the inequality~\eqref{eqn:alphafreedom4} lets us pick $\alpha, \beta$ with $\beta = \alpha+g$ for some $g \in \{1,2,3,6\}$ such that $3 | 2(\alpha\beta-m) -2$ and $\alpha\beta-m$ is odd. Specifically, if $m$ is odd, pick appropriate $\alpha$ and $\beta = \alpha + g$ for $g\in \{1,3\}$, while if $m$ is even pick $\beta = \alpha + g$ with $g\in \{2,6\}$. 

By considering the conditions~\eqref{eqn:alphafreedom4} and~\eqref{eqn:avoid8m+4}, we can now successfully determine a lower bound $m_0$ such that $\calM_{8m+2}$ is of general type for $m \ge m_0$. 
First, note that for $\alpha$, $\beta = \alpha+g$, and $m$ satisfying~\eqref{eqn:inequality}, we have the inequality
\begin{equation}\label{eqn:alphaepislon4} \alpha\beta - m = \alpha^2 + g\alpha - m > \alpha^2 - m > \rho m
\end{equation}
and, as an immediate consequence,
$$2(\alpha\beta - m) -28 > 2\rho m-28.$$ 
Thus, taking 
\begin{equation}
\label{eqn:epsilon4}
\rho m>52
\end{equation} 
will ensure that $2(\alpha \beta -m)-28>2$.  If $2(\alpha \beta -m)-28= N$, where $N$ is as defined in Lemma~\ref{lem:threesquares}, then the inequalities 
$$\alpha \beta - m = \beta^2 - g\beta -m <\beta^2-m< \frac{m}{4}  $$
hold under our continuing assumptions on $\alpha, \beta = \alpha+g$, and $m$. Therefore for such $N$ we must have
$$N <m/2 - 28.$$
So we would like to ensure that for $m >2(N+28)$, the quantity $\rho>0$ is small enough so that the difference 
$$\sqrt{\frac{5m}{4}} - \sqrt{(1+\rho)m}$$
is large enough to adjust $\alpha, \beta$ by $\pm 6$ (to preserve~\eqref{eqn:3divis4} and~\eqref{eqn:1mod2}) in order to avoid $N$.

As before, optimization for~\eqref{eqn:alphafreedom4} and~\eqref{eqn:epsilon4} yields $m \ge 10463$ and $\rho = 0.0014337$. 
\begin{equation}
\sqrt{\frac{5m}{4}} - \sqrt{(1+\rho)m} > 50000
\end{equation}
so we are always able to adjust $\alpha$ to avoid $N$. As in~\S\ref{ss:d8m},we now have proven that Proposition~\ref{prop:mainprop} is true for $m \ge 10772$. The remaining cases for $d=8m+4$ are handled by computer (see~\S\ref{ss:computersearch}).

\subsection{Searching for embeddings by computer}
\label{ss:computersearch}
A list of embeddings for the values of $m$ less than the lower bounds we calculated above is available on the author's webpage. To find these embeddings, we used a simple transplantation of the algorithm given in~\cite[\S 5]{TVACrelle}.  Our search for these embeddings was exhaustive:  we include in our list every $m$ for which there exists an embedding $K_d^\perp \to L_{2,26}$ with our desired properties. We include this list along with {\tt Magma} code~\cite{Magma} to certify that the embeddings in our list produce modular forms of the correct weight\footnote{The list and code are available at~\url{http://math.dartmouth.edu/~jpetok/KodairaCode.m} and in v1 of the arXiv version of this article}. To count the size of $R_{-2}$ corresponding for each embedding, we count by their Type from Lemma~\ref{lem:types} (see Step (iv) of the algorithm in~\cite[\S 5]{TVACrelle}). Our list of explicit embeddings, taken together with the analyses in~\S\S\ref{ss:d8m},~\ref{ss:d8m+2},~\ref{ss:d8m+4}, prove Proposition~\ref{prop:mainprop}.

\end{proof}

\subsection{Acknowledgements}
We are grateful to Sho Tanimoto for introducing us to the problem. We also wish to thank Brendan Hassett for useful discussions at the Simons Collaboration Conference on Arithmetic Geometry, Number Theory, and Computation held in Cambridge, MA in 2018. We thank Emma Brakkee for help with Proposition~\ref{prop:threequotients}  and Jennifer Berg for many useful conversations at Rice University. We also thank Shouhei Ma and Tonghai Yang for helpful email correspondence. We are grateful to the anonymous referees for their comments on earlier versions of the paper. Finally, we wish to thank Anthony V\'{a}rilly-Alvarado (the author's Ph.D. thesis advisor) for his invaluable guidance and support.

While completing this work, the author was supported by NSF grants DMS-1745670 and DMS-1902274. Some of this work was also completed at the Institut Henri Poincar\'{e} during the semester-long program ``Reinventing Rational Points''. The author wishes to thank to organizers of that program for providing a welcoming and productive research environment.

\end{document}